\documentclass[11 pt,dvipsnames]{article}
\usepackage{graphicx}
\usepackage{camnum}
\usepackage{amsmath}
\usepackage{amsfonts}
\usepackage{epstopdf}
\usepackage{xcolor}
\usepackage{harvard}

\numberwithin{figure}{section}
\numberwithin{equation}{section}
\numberwithin{table}{section}
\numberwithin{theorem}{section}

\allowdisplaybreaks

\newcommand{\EE}{\MM{E}}

\renewcommand{\SS}{\MM{S}}

\newcommand{\ee}{{\mathrm e}}
\newcommand{\ii}{{\mathrm i}}

\newcommand{\LT}{\mathrm{LT}}

\newcommand{\St}{\mathrm{S}}
\newcommand{\LL}{\mathcal{L}}

\date{}

\begin{document}

\title{Splitting methods for unbounded operators}

\author{Arieh Iserles\\
Department of Applied Mathematics and Theoretical Physics\\
University of Cambridge\\
Cambridge CB3 0WA, United Kingdom \and
Karolina Kropielnicka\\
Institute of Mathematics of Polish Academy of Sciences\\
18, Abrahama Street, Sopot 81-825, Poland}

\maketitle

\thispagestyle{empty}
\begin{abstract}
  
  This paper considers computational methods that split a vector field into three components in the case when both the vector field and the split components might be unbounded. We first employ classical Taylor expansion which, after some algebra, results in an expression for a second-order splitting which, strictly speaking, makes sense only for bounded operators. Next, using an alternative approach, we derive an error expression and an error bound in the same setting which are however  valid in the presence of unbounded operators. 
  
  While the paper itself is concerned with second-order splittings using three components, the method of proof in the presence of unboundedness remains valid (although significantly more complicated) in a more general scenario, which will be the subject of a forthcoming paper.
\end{abstract}

\noindent{\bf Keywords:} Operator splitting, unbounded operators, Duhamel's formula, error bounds.

\noindent{\bf AMS (MOC) Subject Classification:} Primary 65J10, Secondary 65M15.


\section{The problem}

Splitting methods are a major methodology in modern numerical analysis of time-dependent differential equations \cite{blanes24smd,mclachlan02sm}. Very often we wish to solve a problem of the form
\begin{equation}
  \label{eq:1.1}
  \frac{\D u}{\D t}=f(u)+g(u),\qquad t\geq t_0,
\end{equation}
given with an initial condition at $t=t_0$ and possibly suitable boundary conditions, while the solution of the problems
\begin{equation}
  \label{eq:1.2}
  \frac{\D v}{\D t}=f(v)\qquad\mbox{and}\qquad \frac{\D w}{\D t}=g(w)
\end{equation}
is either known or can be approximated with great ease. The idea of a splitting method is to compose an approximation to the solution of \R{eq:1.1} by leveraging our ability to solve \R{eq:1.2}. In particular, once both $f$ and $g$ are linear, $f(u)=Au$ and $g(u)=Bu$ (where $A$ and $B$ are, in general, linear operators), the exact solution of \R{eq:1.1}, $u(t_0+h)=\ee^{h\LL}u(t_0)$ where $\LL=A+B$, might be difficult to compute and popular approaches are the {\em Lie--Trotter splitting\/}
\begin{equation}
  \label{eq:1.3}
  u(t_0+h)\approx u_{\LT}(t_0+h)=\ee^{hA}\ee^{hB}u(t_0)
\end{equation}
and the {\em Strang splitting\/}
\begin{equation}
  \label{eq:1.4}
  u(t_0+h)\approx u_{\St}(t_0+h)=\ee^{\frac12 hA}\ee^{hB}\ee^{\frac12 hA}u(t_0).
\end{equation}

Provided that $A$ and $B$, hence also $\LL$, are bounded linear operators (for example, finite-dimensional matrices), we can expand in Taylor series and prove with great ease that
\begin{displaymath}
  u_{\LT}(t_0+h)-u(t_0+h)+\O{h^2},\qquad u_{\St}(t_0+h)=u(t_0+h)+\O{h^3}
\end{displaymath}
-- the Lie--Trotter and Strang splittings are of {\em order} 1 and 2, respectively. Likewise, it is possible to derive higher-order splittings, cf.\ \cite{blanes24smd,mclachlan02sm}. This, however, becomes significantly more complicated once either $A$ or $B$ (or both) are unbounded operators, since Taylor expansion may no longer make sense . An example, that motivated much interest in this problem, is the {\em linear Schr\"odinger equation\/}
\begin{equation}
  \label{eq:1.5}
  \ii\frac{\partial u}{\partial t}=\frac12 \Delta u-V(x)u,\qquad x\in\mathbb{R}^d,
\end{equation}
given with the initial condition $u(x,0)=u_0\in\CC{H}^1(\mathbb{R}^d)$. Here $V$ is the {\em interaction potential\/} of a quantum system, $V:\mathbb{R}^d\rightarrow\mathbb{R}$. Note that, while $V$ is a multiplication operator, which is bounded,\footnote{Its boundedness is contingent on $V$ being a bounded function. There are instances, like the evolution of a hydrogen atom in quantum mechanics, where $V$ is unbounded.} the Laplace operator $\Delta$ is unbounded. Thus, while  $\ii t(-\frac12 \Delta+V)$ is unbounded  it may generate  a  semigroup  $\ee^{\ii t(-\frac12 \Delta+V)}$ that is strongly continuous, but not necessarily have a Taylor expansion. The same applies to  $\ee^{\ii\alpha t\Delta}$ and $\ee^{\ii\alpha tV}$ for any $\alpha\in\mathbb{R}$. Therefore we cannot be assured that, say, the Lie--Trotter splitting \R{eq:1.3} or the Strang splitting \R{eq:1.4} are of order 1 or 2, respectively, just by using a Taylor expansion.

Another way of looking at the problem is that unboundedness places additional restrictions on the function class under consideration. Referring again to \R{eq:1.5} and restricting the discussion to classical solutions, we need $u\in\CC{H}^2(\mathbb{R}^d)$ for $\Delta u$ to make sense and $u\in\CC{H}^{2s}(\mathbb{R}^d)$ for the approximation $ \ee^{\ii t\Delta}u\approx \sum_{m=0}^s \frac{1}{m!} (\ii t\Delta)^m u$ to be well defined. This means regularity well in excess of the minimal existence and uniqueness conditions for a classical solution of \R{eq:1.5}. 

A good illustration of limitations on regularity is the function
\begin{displaymath}
  u(x)=
  \begin{case}
    \ee^{-x^\alpha}, & x>0,\\[4pt]
    0, & x\leq0,
  \end{case}
\end{displaymath}
where $\alpha>0$ is non-integer. It is elementary to prove that $u\in\CC{H}^{s}(0,\infty)$ if and only if $\alpha>s-\frac12$. (If $\alpha\in\mathbb{N}$ then $u\in\CC{C}^\infty(0,\infty)$.)

The minimal conditions consistent with an order-$p$ splitting require $u\in\CC{H}^p$ (on the face of it, it might appear that we need $u\in\CC{H}^{p+1}$, to ensure that the error term is bounded, but we will see in the sequel that this is not necessary for $p=2$ and it is safe to conjecture that this is the general state of affairs). The challenge is to close the gap between $p$ and $2p$, specifically to prove that $u\in\CC{H}^p$, in tandem with formal order conditions, is sufficient for order $p$ {\em and\/} for an explicit form of the error. 

Splittings of unbounded operators have been subject to significant research effort \cite{hansen09esu,jahnke00ebe,kropielnicka03fst}. Recently, the present authors introduced  an elementary approach allowing for order analysis and the derivation of tight error bounds for three basic splittings, Lie--Trotter, parallel Lie--Trotter and Strang \cite{iserles04eas}. The purpose of this paper is to sketch an extension of this approach to more general splittings, an endeavour replete with challenges yet allowing eventual order and error analysis using fairly elementary means.

The splittings investigated in this paper are of the form
\begin{equation}
  \label{eq:1.6}
  \SS(t;P_1,\ldots,P_n):=\ee^{tP_1}\ee^{tP_2}\cdots\ee^{tP_n}\approx\ee^{t(P_1+P_2+\cdots+P_n)},
\end{equation}
where $\LL, P_1,\ldots,P_n$ are linear operators, which might be unbounded, and $\sum_{k=1}^n P_k=\LL$. We say that a splitting is {\em canonical\/} if $\LL=A+B$, $[A,B]\neq O$ and each $P_\ell$ is either a scalar multiple of $A$ or a scalar multiple of $B$: the Lie--Trotter splitting \R{eq:1.3} and the Strang splitting \R{eq:1.4} are both canonical.\footnote{While a canonical splitting is defined formally also when $[A,B]=O$, in that case it bears no error, somehow missing the entire point of this paper.} In the sequel we do not require the splitting to be canonical unless otherwise stated.

The mechanism that allows to reduce regularity requirements for splittings is implicit in more general rules pertaining to commutators of differential operators. In our case $A=-\frac12\ii\Delta$ (unbounded) and $B=\ii V$ (bounded) -- but
\begin{displaymath}
  [A,B]=-\frac12  \Delta V+\frac12 (\nabla V)\cdot\nabla
\end{displaymath}
is a bounded operator in $\CC{H}^1$, while
\begin{displaymath}
  [B,[A,B]]=\frac12\ii (\nabla V)^\top(\nabla V)
\end{displaymath} 
is a multiplication operator in $\CC{H}^0=\CC{L}_2$. This is just the simplest instance of a more general state of affairs \cite{bader14eaf,chodosh11imr,faou24dsp}. Thus, let $\GG{F}=\GG{F}(A,B)$ be the free Lie algebra generated by the `letters' $A$ and $B$. Given $T\in\GG{F}$, it is composed of, say, $m\geq0$ $A$s and $n\geq0$ $B$s. If $n\geq m+2$ then necessarily $T\equiv0$, otherwise $Tu$ is bounded for $u\in\CC{H}^{m+1-n}(\mathbb{R}^d)$ (and, if $n=m+1$ is a multiplication operator). This is important because, as will be apparent in the sequel, the error expansion of a splitting \R{eq:1.6} can be expressed using commutators in $\GG{F}(P_1,\ldots,P_n)$.

In Section~2 we consider the case $n=3$ using canonical approach, expanding into Taylor series. This turns out to be surprisingly subtle. While establishing  second-order conditions and writing the leading error term is in principle easy, its naive form adds little to our insight because it is seemingly not in $\GG{F}(P_1,P_2,P_3)$. It takes significant effort to show that actually it is composed solely of commutators. Second shortcoming of this approach is that, if any of the $P_\ell$s is unbounded, it tells us little about the error
\begin{displaymath}
  \EE(t)=\SS(t;P_1,\ldots,P_n)-\ee^{t(P_1+P_2+\cdots+P_n)}
\end{displaymath}
or even if it is bounded. This motivates the work of Section~3, where we employ a different approach, expanding upon the work of \cite{iserles04eas}, which allows unboundedness while deriving second-order conditions and an upper bound of the error.

Note that $n=3$ allows for order 2 at most: if we wish order 3 with a canonical splitting we need at least $n=7$. We hope to revisit more general splittings, employing the approach of Section~3, in a future paper. Note further that according to a theorem of Sheng \cite{blanes24smd} any splitting of order greater than two must have some negative coefficients. This renders such splittings of little utility in the solution of diffusive phenomena, which are well posed only in positive direction. (The problem can be overcome by using complex-valued steps \cite{hansen09hos}, an approach that we do not pursue in this paper.)

\section{Classical approach}

We let $\LL=A+B$ be the vector field and split $\LL=P_1+P_2+P_3$. In this section we assume that the $P_\ell$s are bounded. Straightforward Taylor expansion (paying heed to non-commutativity of the $P_\ell$s) yields
\begin{displaymath}
  \ee^{tP_1}\ee^{tP_2}\ee^{tP_3}-\ee^{t(A+B)}=E_2t^2+E_3t^3+\O{t^4},
\end{displaymath}
where
\begin{displaymath}
  E_2=\frac12 ([P_1,P_2]+[P_1,P_3]+[P_2,P_3]),
\end{displaymath}
therefore a second-order condition is
\begin{equation}
  \label{eq:2.1}
  [P_1,P_2]+[P_1,P_3]+[P_2,P_3]=O.
\end{equation}
More effort is required to compute $E_3$,
\begin{Eqnarray*}
  E_3&=&\frac13[P_1,[P_1,P_2]]+\frac16 [P_2,[P_1,P_2]]+\frac13 [P_1,[P_1,P_3]]+\frac16[P_3,[P_1,P_3]]\\
  &&\mbox{}+\frac13[P_2,[P_2,P_3]]+\frac16[P_3,[P_2,P_3]]+\frac16[P_1,[P_2,P_3]]-\frac16[P_3,[P_1,P_2]]\\
  &&\mbox{}+\frac12[P_1,P_2]P_1+\frac12[P_1,P_2]P_2+\frac12[P_1,P_3]P_1+\frac12[P_1,P_3]P_3+\frac12[P_2,P_3]P_2\\
  &&\mbox{}+\frac12[P_2,P_3]P_3+\frac12(P_1P_2P_3-P_3P_2P_1).
\end{Eqnarray*}
This seems as an unwelcome state of affairs, because of the presence of terms of the form $[P_k,P_\ell]P_j\not\in\GG{F}(P_1,P_2,P_3)$. This, however, can be mended once we take \R{eq:2.1} into consideration. Replacing
\begin{Eqnarray*}
  [P_1,P_2]P_1&=&-([P_1,P_3]+[P_2,P_3])P_1,\\{}
  [P_1,P_2]P_2&=&-([P_1,P_3]P_2+[P_2,P_3]P_2,\\{}
  [P_1,P_3]P_3&=&-([P_1,P_2]P_3+[P_2,P_3]P_3,
\end{Eqnarray*}
while leaving all other expressions of the form $[P_k,P_\ell]P_j$ intact, leads to a welter of cancellations and, after elementary algebra,
\begin{Eqnarray*}
  E_3&=&\frac13[P_1,[P_1,P_2]]+\frac16 [P_2,[P_1,P_2]]+\frac13 [P_1,[P_1,P_3]]+\frac16[P_3,[P_1,P_3]]\\
  &&\mbox{}+\frac13[P_2,[P_2,P_3]]+\frac16[P_3,[P_2,P_3]]+\frac16[P_1,[P_2,P_3]]-\frac16[P_3,[P_1,P_2]]\\
  &&\mbox{}+\frac12[P_2,[P_1,P_3]].
\end{Eqnarray*}
Thus, $E_3$ can be expressed using terms in $\GG{F}(P_1,P_2,P_3)$. Yet, it can be further simplified a great deal. Since
\begin{Eqnarray*}
  &&\frac13 [P_1,[P_1,P_2]]+\frac13[P_1,[P_1,P_3]]+\frac16[P_1,[P_2,P_3]]\\
  &=&\frac13 [P_1,[P_1,P_2]+[P_1,P_3]+[P_2,P_3]]-\frac16 [P_1,[P_2,P_3]]=-\frac16[P_1,[P_2,P_3]],\\
  &&\frac16[P_2,[P_1,P_2]]+\frac13 [P_2,[P_2,P_3]]+\frac12[P_2,[P_1,P_3]]\\
  &=&\frac13[P_2,[P_1,P_2]+[P_1,P_3]+[P_2,P_3]]-\frac16 [P_2,[P_1,P_2]]+\frac16 [P_2,[P_1,P_3]]\\
  &=&-\frac16 [P_2,[P_1,P_2]]+\frac16 [P_2,[P_1,P_3]]\\
  &&\frac16[P_3,[P_1,P_3]]+\frac16[P_2,[P_2,P_3]]-\frac16[P_3,[P_1,P_2]]\\
  &=&\frac16[P_3,[P_1,P_2]+[P_1,P_3]+[P_2,P_3]]-\frac13 [P_3,[P_1,P_2]]=-\frac13 [P_3,[P_1,P_2]],
\end{Eqnarray*}
using the second-order condition \R{eq:2.1} yields
\begin{Eqnarray*}
  E_3&=&-\frac16[P_1,[P_2,P_3]]-\frac16 [P_2,[P_1,P_2]]+\frac16 [P_2,[P_1,P_3]]-\frac13 [P_3,[P_1,P_2]].
\end{Eqnarray*}
Finally, we use the Jacobi identity 
\begin{displaymath}
  [P_2,[P_1,P_3]]=[P_1,[P_2,P_3]]+[P_3,[P_1,P_2]]
\end{displaymath}
and the outcome is
\begin{equation}
  \label{eq:2.2}
  E_3=-\frac16[P_2,[P_1,P_2]]-\frac16 [P_3,[P_1,P_2]].
\end{equation}

\section{An iterated Duhamel approach}

We let
\begin{displaymath}
  \SS(t)=\ee^{tP_1}\ee^{tP_2}\ee^{tP_3},\qquad \EE(t)=\ee^{t(P_1+P_2+P_3)}-\MM{S}(t),
\end{displaymath}
subject to the {\em consistency condition\/} $P_1+P_2+P_3=\mathcal{L}$ and seek both order-2 conditions and an explicit expression for the error $\EE$. The operators $P_\ell$ might be unbounded. This section is devoted to the proof of the main result of this paper:

\begin{theorem}
  \label{thm:3rdOrder}
  Subject to the consistency condition $P_1+P_2+P_3=\LL$, the splitting $\SS$ is of order three if and only if 
  \begin{displaymath}
  [P_1,P_2],[P_1,P_3],[P_2,P_3],[P_1,[P_2,P_3]],[P_2,[P_2,P_3]]
  \end{displaymath}
  are bounded and the condition \R{eq:2.1}, namely
  \begin{displaymath}
    [P_1,P_2]+[P_1,P_3]+[P_2,P_3]=O,
  \end{displaymath}
  holds. In that case the error committed by the splitting is
  \begin{Eqnarray}
    \nonumber
    \MM{E}(t)&=&\int_0^t \ee^{(t-\tau)\mathcal{L}} \int_0^\tau \ee^{(\tau-\eta)P_1} \int_0^\eta \ee^{\xi P_1}[P_1,[P_2,P_3]]\ee^{-\xi P_1}\D\xi \ee^{\eta P_1}\D\eta \ee^{\tau P_2}\ee^{\tau P_3}\D\tau\hspace*{10pt} \\
    \label{eq:3.1}
    &&\mbox{}+\int_0^t \ee^{(t-\tau)\mathcal{L}} \int_0^\tau \ee^{\tau P_1} \int_0^\eta \ee^{\xi P_2}[P_2,[P_2,P_3]]\ee^{-\xi P_2}\D\xi \D\eta  \ee^{\tau P_2}\ee^{\tau P_3}\D\tau.
  \end{Eqnarray}
\end{theorem}

The remainder of this section is concerned with the proof of Theorem~\ref{thm:3rdOrder}. Our approach is based upon  application of the familiar {\em Duhamel principle,\/} namely that the solution of $\MM{y}'=C\MM{y}+\MM{d}(t,\MM{y})$, $\MM{y}(0)=\MM{y}_0$, obeys
\begin{displaymath}
  \MM{y}(t)=\ee^{tC}\MM{y}_0+\int_0^t \ee^{(t-\tau)C} \MM{d}(\tau,\MM{y}(\tau))\D\tau.
\end{displaymath}

Computing directly,
\begin{displaymath}
  \MM{S}'-(P_1+P_2+P_3)\MM{S}=[\ee^{tP_1},P_2]\ee^{tP_2}\ee^{tP_3}+[\ee^{tP_1}\ee^{tP_2},P_3]\ee^{tP_3}.
\end{displaymath}
Therefore $\MM{S}(0)=\MM{I}$, $\LL=P_1+P_2+P_3$ and the Duhamel principle imply, after easy algebra, 
\begin{Eqnarray*}
  \MM{E}(t)=\MM{S}(t)-\ee^{t\mathcal{L}}&=&\int_0^t \ee^{(t-\tau)\mathcal{L}} [\ee^{\tau P_1},P_2]\ee^{\tau P_2}\ee^{\tau P_3}\D\tau\\
  &&\mbox{}+\int_0^t \ee^{(t-\tau)\mathcal{L}} [\ee^{\tau P_1}\ee^{\tau P_2},P_3]\ee^{\tau P_3}\D\tau.
\end{Eqnarray*}

\begin{proposition}
  \label{prop:Z}
  Let
  \begin{displaymath}
    Z(t;P,Q)=[\ee^{tP},Q],
  \end{displaymath}
  where $P$ and $Q$ are operators, which need not be bounded, except that $P$ generates a strongly continuous semigroup, while $[P,Q]$ is bounded. Then
  \begin{equation}
    \label{eq:3.2}
    Z(t)=\ee^{tP}\int_0^t \ee^{-\tau P}[P,Q]\ee^{\tau P}\D\tau=\int_0^t \ee^{\tau P}[P,Q]\ee^{-\tau P}\D\tau \ee^{tP}.
  \end{equation}
\end{proposition}

\begin{proof}
  We first note that $Z(0)=O$. Since
  \begin{displaymath}
    \frac{\D Z(t;P,Q)}{\D t}-PZ(t;P,Q)=[P,Q]\ee^{tP},
  \end{displaymath}
  \R{eq:3.2} follows by the Duhamel formula \R{eq:3.1}.
\end{proof}

Let next
\begin{displaymath}
  V(\tau)=[\ee^{\tau P_1}\ee^{\tau P_2},P_3].
\end{displaymath}
Then
\begin{displaymath}
  V(\tau)=\ee^{\tau P_1}[\ee^{\tau P_2},P_3]+[\ee^{\tau P_1},P_3]\ee^{\tau P_2}=\ee^{\tau P_1}Z(\tau;P_2,P_3)+Z(\tau;P_1,P_3)\ee^{\tau P_2}.
\end{displaymath}
Therefore
\begin{Eqnarray*}
  \MM{E}(t)&=&\int_0^t \ee^{(t-\tau)\mathcal{L}} Z(\tau;P_1,P_2)\ee^{\tau P_2}\ee^{\tau P_3}\D\tau+\int_0^t \ee^{(t-\tau)\mathcal{L}} \ee^{\tau P_1} Z(\tau;P_2,P_3)\ee^{\tau P_3}\D\tau\\
  &&\mbox{}+\int_0^t \ee^{(t-\tau)\mathcal{L}} Z(\tau;P_1,P_3)\ee^{\tau P_2}\ee^{\tau P_3}\D\tau
\end{Eqnarray*}
and it follows from \R{eq:3.2} that
\begin{Eqnarray*}
   \MM{E}(t)&=&\int_0^t \ee^{(t-\tau)\mathcal{L}}\int_0^\tau \ee^{\eta P_1}[P_1,P_2]\ee^{-\eta P_1}\D\eta \MM{S}(\tau)\D\tau \\
   &&\mbox{}+\int_0^t \ee^{(t-\tau)\mathcal{L}} \ee^{\tau P_1} \int_0^\tau \ee^{\eta P_2}[P_2,P_3]\ee^{-\eta P_2}\D\eta \ee^{\tau P_2}\ee^{\tau P_3}\D\tau \\
   &&\mbox{}+\int_0^t \ee^{(t-\tau)\mathcal{L}} \int_0^\tau \ee^{\eta P_1}[P_1,P_3]\ee^{-\eta P_1}\D\eta \MM{S}(\tau)\D\tau.
\end{Eqnarray*}
Let
\begin{displaymath}
  W(\tau)=\ee^{\tau P_1}\int_0^\tau \ee^{\eta P_2}[P_2,P_3]\ee^{-\eta P_2}\D\eta-\int_0^\tau \ee^{\eta P_1}[P_2,P_3]\ee^{-\eta P_1}\D\eta \ee^{\tau P_1},
\end{displaymath}
therefore
\begin{Eqnarray*}
   \MM{E}(t)&=&\int_0^t \ee^{(t-\tau)\mathcal{L}}\int_0^\tau \ee^{\eta P_1}[P_1,P_2]\ee^{-\eta P_1}\D\eta \MM{S}(\tau)\D\tau \\
   &&\mbox{}+\int_0^t \ee^{(t-\tau)\mathcal{L}}  \int_0^\tau \ee^{\eta P_1}[P_2,P_3]\ee^{-\eta P_1}\D\eta \MM{S}(\tau)\D\tau \\
   &&\mbox{}+\int_0^t \ee^{(t-\tau)\mathcal{L}} \int_0^\tau \ee^{\eta P_1}[P_1,P_3]\ee^{-\eta P_1}\D\eta \MM{S}(\tau)\D\tau\\
   &&\mbox{}+\int_0^t \ee^{(t-\tau)\mathcal{L}} W(\tau)\ee^{\tau P_2}\ee^{\tau P_3}\D\tau\\
   &=&\int_0^t \ee^{(t-\tau)\mathcal{L}}\int_0^\tau \ee^{\eta P_1}([P_1,P_2]+[P_1,P_3]+[P_2,P_3])\ee^{-\eta P_1}\D\eta \MM{S}(\tau)\D\tau\\
   &&\mbox{}+\int_0^t \ee^{(t-\tau)\mathcal{L}} W(\tau)\ee^{\tau P_2}\ee^{\tau P_3}\D\tau.
\end{Eqnarray*}
Enforcing the second-order condition \R{eq:2.1} we thus obtain
\begin{displaymath}
  \MM{E}(t)=\int_0^t \ee^{(t-\tau)\mathcal{L}} W(\tau)\ee^{\tau P_2}\ee^{\tau P_3}\D\tau.
\end{displaymath}

The next step is to rewrite $W$ in a more user-friendly form. We have $W(0)=O$ and, after elementary algebra,
\begin{Eqnarray}
  \label{eq:3.3}
  W'(\tau)-P_1W(\tau)&=&\left[P_1,\int_0^\tau \ee^{\eta P_1}[P_2,P_3]\ee^{-\eta P_1}\D\eta\right]\!\ee^{\tau P_1}\\
  \nonumber
  &&\mbox{}+\ee^{\tau P_1} \left(\ee^{\tau P_2}[P_2,P_3]\ee^{-\tau P_2}-[P_2,P_3]\right)\!.
\end{Eqnarray}
But
\begin{displaymath}
  \left[P_1,\int_0^\tau \ee^{\eta P_1}[P_2,P_3]\ee^{-\eta P_1}\D\eta\right]=\int_0^\tau \ee^{\eta P_1}[P_1,[P_2,P_3]]\ee^{-\eta P_1}\D\eta.
\end{displaymath}
Let 
\begin{displaymath}
  R(\tau)=\ee^{\tau P_2}[P_2,P_3]\ee^{-\tau P_2}-[P_2,P_3].
\end{displaymath}
Therefore
\begin{Eqnarray*}
  R(\tau)&=&[\ee^{\tau P_2},[P_2,P_3]]\ee^{-\tau P_2}=Z(\tau;P_2,[P_2,P_3])\ee^{-\tau P_2} \\
  &=&\int_0^\tau \ee^{\eta P_2}[P_2,[P_2,P_3]]\ee^{-\eta P_2}\D\eta .
\end{Eqnarray*}
Substituting into \R{eq:3.3}, we thus have
\begin{Eqnarray*}
  W'(\tau)-P_1W(\tau)&=&\int_0^\tau \ee^{\eta P_1}[P_1,[P_2,P_3]]\ee^{-\eta P_1}\D\eta \ee^{\tau P_1}\\
  &&\mbox{}+\ee^{\tau P_1}\int_0^\tau \ee^{\eta P_2}[P_2,[P_2,P_3]]\ee^{-\eta P_2}\D\eta .
\end{Eqnarray*}
Since $W(0)=O$, Duhamel yields
\begin{Eqnarray*}
  W(\tau)&=&\int_0^\tau \ee^{(\tau-\eta)P_1} \int_0^\eta \ee^{\xi P_1}[P_1,[P_2,P_3]]\ee^{-\xi P_1}\D\xi \ee^{\eta P_1}\D\eta\\
  &&\mbox{}+\int_0^\tau \ee^{\tau P_1} \int_0^\eta \ee^{\xi P_2}[P_2,[P_2,P_3]]\ee^{-\xi P_2}\D\xi \D\eta.
\end{Eqnarray*}
Consequently,
\begin{Eqnarray}
  \nonumber
  \MM{E}(t)&=&\int_0^t \ee^{(t-\tau)\mathcal{L}} \int_0^\tau \ee^{(\tau-\eta)P_1} \int_0^\eta \ee^{\xi P_1}[P_1,[P_2,P_3]]\ee^{-\xi P_1}\D\xi \ee^{\eta P_1}\D\eta \ee^{\tau P_2}\ee^{\tau P_3}\D\tau\hspace*{10pt} \\
  \label{eq:3.4}
  &&\mbox{}+\int_0^t \ee^{(t-\tau)\mathcal{L}} \int_0^\tau \ee^{\tau P_1} \int_0^\eta \ee^{\xi P_2}[P_2,[P_2,P_3]]\ee^{-\xi P_2}\D\xi \D\eta  \ee^{\tau P_2}\ee^{\tau P_3}\D\tau.
\end{Eqnarray}

There is still a gap between this and \R{eq:3.1}, as well as a discrepancy (which turns out to be illusory) between the above form of $\EE$ and the fact that in the previous section the leading error term consisted of a linear combination of $[P_2,[P_1,P_2]]$ and $[P_3,[P_1,P_2]]$, while now it is $[P_1,[P_2,P_3]]$ and $[P_2,[P_2,P_3]]$. However,  by \R{eq:2.1}
\begin{displaymath}
  [P_2,[P_1,P_2]]=-[P_2,[P_1,P_3]]-[P_2,[P_2,P_3]],
\end{displaymath}
while the Jacobi identity results in 
\begin{displaymath}
  [P_2,[P_1,P_3]]-[P_3,[P_1,P_2]]=[P_1,[P_2,P_3]].
\end{displaymath}
Therefore
\begin{displaymath}
  E_3=\frac16 ([P_1,[P_2,P_3]]+[P_2,[P_2,P_3]]),
\end{displaymath}
composed of exactly the same building blocks as\R{eq:3.4}. Substitution in \R{eq:3.4} demonstrates that \R{eq:3.1} is true, thereby completing the proof of Theorem~\ref{thm:3rdOrder}.

The error expression \R{eq:3.1} provides a handy way to derive an upper bound of the error committed by the splitting $\SS$ in the case when $\|\ee^{t\mathcal{L}}\|,\|\ee^{tP_\ell}\|\equiv1$, $\ell=1,2,3$, where $\|\,\cdot\,\|$ is a norm. It then follows at once that
\begin{equation}
  \label{eq:3.5}
  \|\MM{E}(t)\|\leq \frac{t^3}{6} (\|[P_1,[P_2,P_3]]\|+\|[P_2,[P_2,P_3]]\|).
\end{equation}
Note that $\|\ee^{t\LL}\|\equiv1$ in the standard Euclidean norm for the linear Schr\"odinger equation \R{eq:1.5}, which has motivated much of this work. 

It is instructive to list few of the important differences between \R{eq:2.2} and \R{eq:3.5}, expressions which, on the face of it, look quite similar. Firstly, the derivation of \R{eq:2.2} was valid only subject to the assumption that all the operators concerned are bounded. Secondly, it depicts only the leading error term and provides no information whatsoever about subsequent terms -- this is important in particular when operators are unbounded and, even subject to the conditions of Theorem~\ref{thm:3rdOrder}, we can make no valid assumptions about the size (and the boundedness) of subsequent terms in the expansion. Finally, \R{eq:3.5} is an upper bound on the error and is hence significantly superior to a leading error term \R{eq:2.2}.

\section*{Acknowledgements}

The work of KK in this project has been supported by The National Center for Science (NCN), based on Grant No.\ 2019/34/E/ST1/00390.

\end{document}